\documentclass[final,leqno,onefignum]{siamltex704}

\usepackage{graphicx,epsfig,latexsym,enumerate}
\usepackage{amsfonts,amssymb,amsmath}
\usepackage{fancyhdr}
\usepackage{showkeys}

\newtheorem{remark}[theorem]{Remark}
\newtheorem{assumption}[theorem]{Assumption}

\newtheorem{algorithm}[theorem]{Algorithm}

\newcommand{\dd}{\mathrm{d}}
\newcommand{\ii}{\mathrm{i}}

\newcommand{\ignore}[1]{}

\newcommand{\norm}[1]{\left\| #1 \right\|}
\newcommand{\notmid}{\mid\kern-0.5em\not\kern0.5em}

\newcommand{\ba}{\begin{array}}
\newcommand{\ea}{\end{array}}

\DeclareMathOperator*{\argmin}{arg\,min}

\numberwithin{equation}{section} 

\title{Inverse boundary value problem for the Helmholtz equation:       Multi-level approach and iterative reconstruction}

\author{Elena Beretta\thanks{Dipartimento di Matematica "Brioschi",
    Politecnico di Milano, Milano and Dipartimento di Matematica
    "Castelnuovo" Universita' di Roma "La Sapienza", Roma, Italy ({\tt
      beretta@mat.uniroma1.it, elena.beretta@polimi.it})}
\and
Maarten V. de Hoop
\thanks{Department of Mathemematics, Purdue
  University, West Lafayette, IN 47907 ({\tt mdehoop@purdue.edu}).}
\and Lingyun Qiu
\thanks{Institute for Mathematics and its Applications ,
  University of Minnesota, Minneapolis, MN 55455 ({\tt qiu.lingyun@ima.umn.edu}).}
\and Otmar Scherzer\thanks{Computational Science Center, University of
  Vienna, Nordbergstr. 15, A-1090 Vienna, Austria ({\tt
    otmar.scherzer@univie.ac.at}).}  }
\begin{document}
\maketitle \setcounter{page}{1}
\cfoot{\thepage} \rfoot{}

\begin{abstract}
We study the inverse boundary value problem for the Helmholtz equation
using the Dirichlet-to-Neumann map at selected frequency as the
data. We develop an explicit reconstruction of the wavespeed using a
multi-level nonlinear projected steepest descent iterative scheme in
Banach spaces. We consider wavespeeds containing discontinuities. A
conditional Lipschitz stability estimate for the inverse problem holds
for wavespeeds of the form of a linear combination of piecewise
constant functions with an underlying domain partitioning, and gives a
framework in which the scheme converges. The stability constant grows
exponentially as the number of subdomains in the domain partitioning
increases. To mitigate this growth of the stability constant, we
introduce hierarchical compressive approximations of the solution to
the inverse problem with piecewise constant functions. We establish an
optimal bound of the stability constant, which leads to a
condition on the compression rate pertaining to these approximations.
\end{abstract}

\maketitle

\section{Introduction}
\label{sec:Intro}

In this paper, we study the inverse boundary value problem for the
Helmholtz equation using the Dirichlet-to-Neumann map at selected
frequency as the data. We focus on developing an explicit iterative
reconstruction of the wavespeed. This inverse problem arises, for
example, in reflection seismology and inverse obstacle scattering
problems for electromagnetic waves \cite{Bao2005a, Symes2009,
  Bao2010}. We consider wavespeeds containing discontinuities.

Uniqueness of the mentioned inverse boundary value problem was
established by Sylvester \& Uhlmann \cite{Sylvester1987} assuming that
the wavespeed is a bounded measurable function. This inverse problem
has been extensively studied from an optimization point of view. We
mention, in particular, the work of \cite{Hadj-Ali2008}. Using
multi-frequency data, to intuitively stabilize the iterative schemes,
so-called frequency progression has been introduced \cite{Bunks1995,
  Sirgue2004, Bao2007, Chen2009}. Frequency progression also appears
naturally in our approach.

It is well known that the logarithmic character of stability of the
inverse boundary value problem for the Helmholtz equation
\cite{Alessandrini1988, Novikov2011} cannot be avoided. In fact, in
\cite{Mandache2001} Mandache proved that despite of regularity or
a-priori assumptions of any order on the unknown wavespeed,
logarithmic stability is optimal. However, conditional Lipschitz
stability estimates can be obtained: Accounting for discontinuities,
such an estimate holds if the unknown wavespeed is a finite linear
combination of piecewise constant functions with an underlying known
domain partitioning \cite{Beretta2012}. We can extend this result to a
framework of discrete approximations of the unique solution. It was
obtained following an approach introduced by Alessandrini and Vessella
\cite{Alessandrini2005} and further developed by Beretta and Francini
\cite{Beretta2011} for Electrical Impedance Tomography (EIT). Here, we
revisit the Lipschitz stability estimate for the full
Dirichlet-to-Neumann map using complex geometrical optics solutions
which gives rise to an optimal bound of the Lipschitz constant in
terms of the number of subdomains in the domain partitioning.  We
develop the estimate in $L^2(\Omega)$. This result aids in the design
of our multi-level iterative scheme for reconstruction through
increasing the number of subdomains. 

To be more precise, we introduce a multi-scale hierarchy of
compressive approximations, while refining the domain partitioning, of
the unique solution to the inverse problem with piecewise constant
functions. Assuming a bound on the compression, we mitigate the
mentioned growth of the stability constant in a multi-level projected
steepest descent method yielding a condition which couples the
approximation errors and stability constants between neighboring
levels \cite{Hoop2012a}. Tracking the domain partitioning dependencies in the
approximation errors we then arrive at a procedure to refine the domain partitioning, which are related to the scales, guaranteeing convergence
from level to level, that is, to progressively more accurate
approximations following the multi-scale hierarchy. As a part of
the analysis, we study the Fr\'{e}chet differentiability of the direct
problem and obtain the frequency and domain partitioning dependencies of the relevant
constants away from the Dirichlet spectrum. Our results hold for
finite-frequency data including frequencies arbitrarily close to zero
while avoiding Dirichlet eigenfrequencies; in view of the estimates,
inherently, there is a finest scale which can be reached.


\subsection*{Frequency data}

In many applications such as reflection seismology, frequency data are
obtained from solutions to the corresponding boundary value problem
for the wave equation by applying a Fourier transform. Let $\Omega$ be
a bounded $C^{1}$ domain in $\mathbb{R}^3$ and $c = c(x)$ be a
strictly positive bounded measurable function.  We consider the
boundary value problem for the wave equation
\[
\left\{
\begin{aligned}
  &\partial_t^2 u(x,t) - c^2(x) \Delta u(x,t) = 0, &&\quad (x,t) \in \Omega \times  (0,\infty) ,
  \\
  &u(x,t) = f(x,t) , &&\quad (x,t) \in \partial\Omega  \times (0,\infty) ,
  \\
  &u(x,0) = 0, \,\, \partial_t u(x,0) = 0 ,&&\quad x\in \Omega .
\end{aligned}
\right.
\]
The hyperbolic Dirichlet-to-Neumann map, $\check{\Lambda}_{c^{-2}}$,
is given by
\[
\begin{aligned}
  \check{\Lambda}_{c^{-2}} :\ H
  & \rightarrow && L^2(\partial \Omega \times (0,\infty)) ,
\\
  f & \mapsto &&
  \partial_\nu u^f \mid_{\partial\Omega \times (0,\infty)} ,
\end{aligned}
\]
where $\partial_\nu$ denotes the normal derivative at $\partial\Omega$
and $H = \{ f \in H^1(\partial \Omega \times (0,\infty)) \mid f(x,0)
= 0 \}$. One, indeed, may take the Fourier transform of $\partial_\nu
u^f$, since it is a tempered distribution \cite{Lasiecka1983}, and
thus obtain multi-frequency data.

\section{Direct problem}
\label{sec:Direct}

We describe the direct problem and some properties of the data, that
is, the Dirichlet-to-Neumann map. We formulate the direct problem as a
nonlinear operator mapping $L^2(\Omega)$ to
$\mathcal{L}(H^{1/2}(\partial \Omega), H^{-1/2}(\partial \Omega))$.
We invoke

\medskip\medskip

\begin{assumption}\label{Aprioribound}
There exist two positive constants $B_1, B_2$ such that
\begin{equation}
   B_1 \leq c^{-2} \leq B_2\quad\text{in }\Omega.
\end{equation}
\end{assumption}

\medskip\medskip

Following the above mentioned Fourier transform, we consider
the boundary value problem,
\begin{equation} \label{Helmholtz}
\left\{\begin{array}{rl}
   (-\Delta - \omega^2 c^{-2}(x)) u = &  0 ,\quad  \mbox{ in }\Omega ,\\
   u = & g \quad \mbox{ on } \partial \Omega.
   \end{array}
   \right.
\end{equation}

\medskip\medskip

We summarize some results concerning the well-posedness of problem
(\ref{Helmholtz}) which we will use in the proofs of the properties of
the Dirichlet-to-Neumann map.

\medskip\medskip

\begin{proposition}\label{2-energy}
Let $\Omega$ be a bounded Lipschitz domain in $\mathbb{R}^3$, $f \in
L^{2}(\Omega)$, $g \in H^{1/2}(\partial\Omega)$ and $c^{-2}\in
L^{\infty}(\Omega)$ satisfying (\ref{Aprioribound}). Then, there
exists a discrete set $\Sigma_{c^{-2}} := \{ \tilde\lambda_n\quad
|\tilde\lambda_n > 0 ,\ \forall n \in \mathbb{N}\}$ such that, for any
$\omega^2 \in \mathbb{C}\backslash \Sigma_{c^{-2}}$, there exists a
unique solution $u \in H^1(\Omega)$ of
\begin{equation} \label{Helmholtz1}
\left\{\begin{array}{rl}
   (-\Delta - \omega^2 c^{-2}(x)) u
                = &  f\ \ \quad  \mbox{ in } \Omega ,\\
   u = & g \quad \mbox{ on } \partial \Omega .
   \end{array}
   \right.
\end{equation}
Furthermore, there exists a positive constant $C$ such that
\begin{equation}\label{energy_est1}
\| u \|_{H^{1}(\Omega)} \le C \left( \| g \|_{H^{1/2}(\partial \Omega)} + \| f \|_{L^{2}(\Omega)} \right) ,
\end{equation}
where  $C=C( B_2, \omega^2, \Omega, d)$ where $d:=\textrm{dist}(\omega^2, \Sigma_{c^{-2}})$ and $C$ blows up as $d\rightarrow 0$.
\end{proposition}

\begin{proof}
We first prove the result for $g=0$. Consider the linear operator
$K=(-\Delta)^{-1}M_{c^{-2}}: L^2\rightarrow H^1_0(\Omega)$ where
$M_{c^{-2}}: L^2(\Omega)\rightarrow L^2(\Omega)$ is the multiplication
operator $h\rightarrow c^{-2}h$.  Then, problem (\ref{Helmholtz1}) is
equivalent to
\begin{equation}\label{eqK}
(I-\omega^2K)u=(-\Delta)^{-1}f =: h .
\end{equation}
We observe that $K$ is a compact operator from $L^2(\Omega)$ to
$L^2(\Omega)$. In fact,
$$
   \|Ku\|_{H_0^1(\Omega)}\leq C\|u\|_{L^2(\Omega)} ,
$$
and by the Rellich-Kondrachev compactness theorem $H^1_0(\Omega)
\subset \subset L^2(\Omega)$.  Furthermore, by Assumption
\ref{Aprioribound} and the properties of $(-\Delta)^{-1}$ and of
$M_{c^{-2}}$ it follows that $K$ is self-adjoint and positive. Hence,
$K$ has a discrete set of positive eigenvalues $\{ \alpha_n
\}_{n\in\mathbb{N}}$ such that $\alpha_n \rightarrow 0$ as $n
\rightarrow \infty$. Let $\tilde\lambda_n := \frac{1}{\alpha_n}, n
\in\mathbb{N}$ and define $\Sigma_{c^{-2}} := \{
\tilde{\lambda}_n\ :\ n \in \mathbb{N}\}$ and let $\omega^2 \in
\mathbb{C} \backslash \Sigma_{c^{-2}}$. Then, by the Fredholm
alternative, there exists a unique solution $u \in H^1_0(\Omega)$ of
(\ref{eqK}).

To prove estimate (\ref{energy_est1}) we observe that
$$
u = \sum_{n=1}^{\infty} \langle u,e_n \rangle e_n ,\quad
K u =\sum_{n=1}^{\infty}\alpha_n \langle u,e_n \rangle e_n
$$
where $\{e_n\}_{n\in\mathbb{N}}$ is an orthonormal basis of $L^2(\Omega)$.
Hence we can rewrite (\ref{eqK}) in the form
$$
\sum_{n=1}^{\infty}(1-\omega^2\alpha_n) \langle u,e_n \rangle e_n
=\sum_{n=1}^{\infty} \langle h_n,e_n \rangle e_n
$$
Hence,
$$
   \langle u,e_n \rangle
    = \frac{1}{1-\frac{\omega^2}{\tilde\lambda_n}}
    \langle h, e_n \rangle,\quad\forall n\in\mathbb{N}
$$
and
$$
   u=\sum_{n=1}^{\infty}
     \frac{1}{1-\frac{\omega^2}{\tilde\lambda_n}}
        \langle h,e_n \rangle e_n
$$
so that
\begin{equation}\label{est_u}
\|u\|_{L^2(\Omega)}\leq \left(1+\frac{\omega^2}{d(\omega^2, \Sigma_{c^{-2}})}\right)\|h\|_{L^2(\Omega)}\leq \left(1+\frac{\omega^2}{d(\omega^2, \Sigma_{c^{-2}})}\right)\|f\|_{L^2(\Omega)}
\end{equation}
Now by multiplying equation (\ref{Helmholtz1}) by $u$, integrating by parts, using Schwartz' inequality, Assumptions (\ref{Aprioribound}) and (\ref{est_u}) we derive
\begin{equation}\label{est_grad}
\|\nabla u\|_{L^2(\Omega)}\leq C\left(1+\frac{\omega^2}{d(\omega^2, \Sigma_{c^{-2}})}\right)\|f\|_{L^2(\Omega)}
\end{equation}
Hence, by (\ref{est_u}) and (\ref{est_grad}) we finally get
$$
\|u\|_{H^1(\Omega)}\leq C\left(1+\frac{\omega^2}{d(\omega^2, \Sigma_{c^{-2}})}\right)\|f\|_{L^2(\Omega)}.
$$
If $g$ is not identically zero then we reduce the problem to the previous case
by considering $v=u-\tilde g$ where $\tilde g\in H^1(\Omega)$ is such
that $\tilde g=g$ on $\partial\Omega$ and $\|\tilde
g\|_{H^1(\Omega)}\leq \|g\|_{H^{1/2}(\partial\Omega)}$.
\end{proof}

If $\omega^2\in\mathbb{C}\backslash \Sigma_{c^{-2}}$, then, by
(\ref{2-energy}), for any $g$ is in $H^{1/2}(\partial \Omega)$, there
exists a unique solution to $(\ref{Helmholtz})$ which belongs to
$H^1(\Omega)$. Therefore, $\nabla u$ is in $L^2(\Omega)$ and as a
consequence $\nabla u \mid_{\partial \Omega}$ belongs to
$H^{-1/2}(\partial \Omega)$. One can then introduce the
Dirichlet-to-Neumann map, $\Lambda_{\omega^2 c^{-2}}$, according to
\begin{equation}\label{DtN map}
\Lambda_{\omega^2 c^{-2}} \, g
 = \nu \cdot \nabla u \mid_{\partial\Omega}
 = \left. \frac{\partial u}{\partial \nu} \right|_{\partial\Omega} \in H^{-1/2}(\partial\Omega).
\end{equation}

Proceeding similarly as in the proof of Proposition~\ref{2-energy} and
using standard regularity results we can prove also the following

\medskip\medskip

\begin{proposition}\label{p-energy}
Let $\Omega$ be a bounded $C^{1,1}$ domain in $\mathbb{R}^3$,
$c^{-2} \in L^{\infty}(\Omega)$ satisfying (\ref{Aprioribound}), $f \in
L^p(\Omega)$, $g \in W^{2-\frac{1}{p},p}(\partial\Omega)$ with $1 < p
< \infty$. Then, if $\omega^2\in\mathbb{C}\backslash \Sigma_{c^{-2}}$,
there exists a unique solution $u\in W^{2,p}(\Omega)$ to the problem
\begin{equation}\label{Helmholtz-L_2}
\left\{\begin{array}{rl}
   (-\Delta - \omega^2 c^{-2}(x)) u
                = &  f\ \ \quad  \mbox{ in } \Omega ,\\
   u = & g \quad \mbox{ on } \partial \Omega .
   \end{array}
   \right.
\end{equation}
Moreover,
\begin{equation}\label{energy_estp}
\| u \|_{W^{2,p}(\Omega)} \le C ( \| g \|_{W^{2-\frac{1}{p},p}(\partial \Omega)} + \| f \|_{L^{p}(\Omega)} )
\end{equation}
where $C=C( B_2, \omega^2, \Omega, d)$ where $d:=\textrm{dist}(\omega^2, \Sigma_{c^{-2}})$ and $C$ blows up as $d\rightarrow 0$.
\end{proposition}

The constants appearing in the estimates of (\ref{2-energy}) and
(\ref{p-energy}) depend on $c^{-2}$ and $ \Sigma_{c^{-2}}$ which are
unknown. To our purposes it would be convenient to have constants
depending only on a priori parameters $B_1,B_2$ and known
parameters.
 In fact, we can prove the following

\medskip\medskip

\begin{proposition}\label{uniformconstants}
Suppose that the assumptions of Proposition \ref{2-energy} (Proposition \ref{p-energy}) are satisfied. Let $\{\lambda_n\}_{n\in\mathbb{N}}$ denote the Dirichlet eigenvalues of $-\Delta$. Then, for any $n\in\mathbb{N}$,
\begin{equation} \label{eigenvaluebound}
\frac{\lambda_n}{B_2}\leq\tilde\lambda_n\leq\frac{\lambda_n}{B_1} .
 \end{equation}
If $\omega^2$ is such that,
 \begin{equation}\label{smallfr}
 0<\omega^2\leq\omega^2_0<\frac{\lambda_1}{B_2},
 \end{equation}
 or, for some $n\geq 1$,
 \begin{equation}\label{higherfr}
 \frac{\lambda_n}{B_1}<\omega_1^2\leq \omega^2\leq \omega_2^2< \frac{\lambda_{n+1}}{B_2} ,
 \end{equation}
then there exists a unique solution $u\in H^1(\Omega)$ $(u\in
W^{2,p}(\Omega)$) of (\ref{Helmholtz}) while estimate
(\ref{energy_est1}) ((\ref{energy_estp})) is satisfied and
$C=C(B_1,B_2, \omega^2_0,\Sigma)$ $(C=C(B_1,B_2,
\omega_1^2,\omega^2_2,\Sigma ))$, where
$\Sigma:=\{\lambda_n\}_{n\in\mathbb{N}}$ and $C$ blows up as
$|\omega_0^2-\frac{\lambda_1}{B_2}|\rightarrow 0$ $(\min
(|\omega_1^2-\frac{\lambda_n}{B_1}|,|\omega_2^2-\frac{\lambda_{n+1}}{B_2}|)
\rightarrow 0)$.
 \end{proposition}

\begin{proof}
To derive estimate (\ref{eigenvaluebound}) we consider the Rayleigh
quotient related to equation (\ref{Helmholtz})
$$
\frac{\int_{\Omega}|\nabla v|^2}{\int_{\Omega}c^{-2}v^2} .
$$
By assumption (\ref{Aprioribound}), for any non trivial $v\in
H^1_0(\Omega)$,
$$
\frac{1}{B_2}\frac{\int_{\Omega}|\nabla v|^2}{\int_{\Omega}v^2}\leq
\frac{\int_{\Omega}|\nabla v|^2}{\int_{\Omega}c^{-2}v^2}\leq
\frac{1}{B_1}\frac{\int_{\Omega}|\nabla v|^2}{\int_{\Omega}v^2} ,
$$
which, by the Courant-Rayleigh minimax principle, immediately gives
$$
\frac{\lambda_n}{B_2}\leq \tilde{\lambda}_n\leq \frac{\lambda_n}{B_1},\quad \forall n\in\mathbb{N}.
$$
Hence, we have well-posedness of problem (\ref{Helmholtz}) if select
an $\omega^2$ satisfying (\ref{smallfr}) or (\ref{higherfr}), and since
$d(\omega^2,\Sigma_{c^{-2}})\geq |\omega_0^2-\frac{\lambda_1}{B_2}|$
or $d(\omega^2,\sigma_{c^{-2}})\geq \min
(|\omega_1^2-\frac{\lambda_n}{B_1}|,|\omega_2^2-\frac{\lambda_{n+1}}{B_2}|)$
the claim follows.
\end{proof}

\medskip\medskip

We observe that in order to derive the uniform estimates in
Proposition \ref{uniformconstants} below, we need to assume either
that the frequency is sufficiently small or that the oscillation of
$c^{-2}$ is sufficiently small. This clearly depends on the fact that
we compare the eigenvalues of our equation with those of the
Laplacian. In the applications we have in mind we aim to recover also
wavespeeds with high contrasts at not too small frequencies. For these
purposes, the following local result will be relevant.

\medskip\medskip

\begin{proposition}\label{continuity}
Let $\Omega$ and $c_0^{-2}$ satisfy the assumptions of Proposition
\ref{2-energy} (\ref{p-energy}) and let
$\omega^2\in\mathbb{C}\backslash\Sigma_{c_0^{-2}}$ where
$\Sigma_{c_0^{-2}}$ is the Dirichlet spectrum of equation
(\ref{Helmholtz}) for $c^{-2}=c_0^{-2}$. There exists $\delta>0$ such
that, if
\[
\|c^{-2}-c_0^{-2}\|_{L^{\infty}(\Omega)}\leq \delta,
\]
then $\omega^2\in \mathbb{C}\backslash\Sigma_{c^{2}}$ and the
estimates (\ref{energy_est1}) ((\ref{energy_estp})) hold, the constant
$C=C( B_2, \omega^2, \Omega, d_0)$ where $d_0:=\textrm{dist}(\omega^2,
\Sigma_{c_0^{-2}})$ and $C$ blows up as $d_0\rightarrow 0$.
\end{proposition}

\begin{proof}
We limit ourselves to prove the well-posedness in $H^1(\Omega)$. Let $\delta c:=c^{-2}-c_0^{-2}$ and consider  $u_0\in H^{1}(\Omega)$ the unique solution of \ref{Helmholtz1} for $c_0^{-2}$ and consider the problem
\begin{equation}\label{auxiliarypr}
\left\{
\begin{array}{rl}
-\Delta v-\omega^2c_0^{-2}v-\omega^2\delta c v= &\omega^2u_0\delta c \quad x\in \Omega, \\
v= & 0, \quad x\in \partial \Omega .
\end{array}
\right.
\end{equation}
Let $L :=-\Delta-\omega^2c_0^{-2}$. Then, by assumption, $L$ is
invertible from $H_0^1(\Omega)$ to $L^2(\Omega)$ and we can rewrite
problem (\ref{auxiliarypr}) in the form
\begin{equation}\label{auxiliaryeq}
(I-K)v=h ,
\end{equation}
where $K=L^{-1}M_{\delta c}$ where $M_{\delta c}$ is the
multiplication operator and $h=L^{-1}(\omega^2u_0\delta c)$. Observe
now that 
\[
\|K\|\leq \|L^{-1}\|\|M_{\delta c}\|\leq \|L^{-1}\| \delta
   \leq B_2 \omega^2(1+\frac{\omega^2}{d_0})\delta.
\]
Hence, choosing $\delta=\frac{1}{2}(
B_2\omega^2(1+\frac{\omega^2}{d_0}))^{-1}$, we have that there exists
a unique solution $v$ of \ref{auxiliaryeq} in $H^1_0$ satisfying
(\ref{energy_est1}) with $C=C (B_2, \omega^2, \Omega, d_0)$ and since
$u=u_0+v$ the statement follows.
\end{proof}

\medskip\medskip

We define the direct operator, $F_\omega$, as
\[
\begin{aligned}
  F_\omega : L^2(\Omega)  & \rightarrow && \mathcal{L}(H^{1/2}(\partial \Omega), H^{-1/2}(\partial \Omega)),
  \\
  c^{-2}(x) & \mapsto && \Lambda_{\omega^2 c^{-2}},
\end{aligned}
\]
and examine its properties in the following lemmas.

\medskip\medskip

\begin{lemma}[Fr\'echet differentiability]\label{Frechet-diff}
Assume that $\omega^2$ satisfies (\ref{smallfr}) or (\ref{higherfr}).
Then the direct operator $F_\omega$ is Fr\'echet differentiable at
$c^{-2}$.
\end{lemma}

\begin{proof}
We start from Alessandrini's identity,
\begin{equation}\label{Alessandrini-iden}
    \int_{\Omega} \omega^2(c_1^{-2} -c_2^{-2}) u_1 u_2 \, \dd x =
    \langle (\Lambda_{\omega^2 c_1^{-2}} - \Lambda_{\omega^2
      c_2^{-2}}) u_1 |_{\partial\Omega} \, ,
                \, u_2 |_{\partial\Omega} \rangle ,
\end{equation}
where $\langle \cdot, \cdot \rangle$ is the dual pairing with respect
to $H^{-1/2}(\partial \Omega)$ and $H^{1/2}(\partial \Omega)$ and
$u_1$ and $u_2$ are the solutions of the Helmholtz equation with
Dirichlet boundary condition and coefficient $c_1$ and $c_2$,
respectively. Let $\delta c^{-2} \in L^\infty(\Omega)$. We observe,
while substituting $c^{-2}$ and $c^{-2} + \delta c^{-2}$ for
$c_1^{-2}$ and $c_2^{-2}$, that
\begin{equation}\label{ident-1}
\langle (\Lambda_{\omega^2(c^{-2} + {\delta c}^{-2})} - \Lambda_{\omega^2 c^{-2}})g \, , \, h \rangle = \omega^2 \int_\Omega {\delta c}^{-2} \, u v \,\dd x ,
\end{equation}
where $u$ and $v$ solve the boundary value problems,
\[
\left\{
\begin{array}{rl}
(-\Delta - \omega^2(c^{-2} + {\delta c}^{-2}))u = & 0, \quad x\in \Omega, \\
u= & g, \quad x\in \partial \Omega,
\end{array}
\right.
\]
and
\[
\left\{
\begin{array}{rl}
(-\Delta - \omega^2 c^{-2})v = & 0, \quad x\in \Omega, \\
v = & h, \quad x\in \partial \Omega,
\end{array}
\right.
\]
respectively. We show that
\begin{equation}\label{Fre-DF}
\langle DF_\omega(c^{-2})({\delta c}^{-2})  g\, , \, h\rangle = \omega^2 \int_\Omega \delta c^{-2} \, \tilde{u} v \,\dd x,
\end{equation}
where $\tilde{u}$ solves the equation
\[
\left\{
\begin{array}{rl}
(-\Delta - \omega^2 c^{-2})\tilde{u} = & 0, \quad x\in \Omega, \\
\tilde{u} = & g, \quad x\in \partial \Omega.
\end{array}
\right.
\]
In fact, by (\ref{ident-1}), we have that
\begin{equation}\label{break-diff-lp}
  \langle (\Lambda_{\omega^2(c^{-2} + {\delta c}^{-2})} - \Lambda_{\omega^2 c^{-2}})g \, , \, h \rangle -  \omega^2\int_\Omega {\delta c}^{-2} \, \tilde{u} v \,\dd x
   = \omega^2 \int_\Omega {\delta c}^{-2} \, (u - \tilde{u}) v \,\dd x.
   %
\end{equation}
By using the H\"older inequality twice and the Sobolev embedding theorem, we obtain that
\begin{equation}\label{ineq:1}
 \left|\omega^2 \int_\Omega {\delta c}^{-2} \, (u - \tilde{u}) v \,\dd x \right|
  \leq \,\, \omega^2 \|{\delta c}^{-2}\|_{L^{2}(\Omega)} \|u - \tilde{u}\|_{L^{4}(\Omega)} \|v\|_{L^{4}(\Omega)}.
\end{equation}
We note that $u-\tilde{u}$ solves the equations
\[
\left\{
\begin{array}{rll}
(-\Delta - \omega^2 c^{-2})(u - \tilde{u}) = & -\omega^2 {\delta c}^{-2} \,  u , & \quad x\in \Omega, \\
u - \tilde{u} = & 0, & \quad x\in \partial \Omega.
\end{array}
\right.
\]
Therefore, by the Sobolev embedding theorem and
Proposition~\ref{uniformconstants}, we find that
\begin{equation}\label{ineq-diff-1}
  \|u - \tilde{u}\|_{L^{4}(\Omega)}
\leq  C \|u - \tilde{u}\|_{W^{2,\frac{12}{11}}(\Omega)}
\leq C \omega^2 \|{\delta c}^{-2} \, u\|_{L^{\frac{12}{11}}(\Omega)} \le C \omega^2 \|{\delta c}^{-2}\|_{L^{2}(\Omega)} \|u\|_{L^{\frac{12}{5}}(\Omega)}.
\end{equation}
The right-most inequality is
obtained by applying the H\"older inequality to $\int_\Omega |{\delta c}^{-2} u|^{12/11} \dd x$ with indexes $11/6$ and $11/5$.
By using the interpolation of $L^p$ spaces,
\[
\|u\|_{L^{p_\theta}(\Omega)} \leq \|u\|_{L^{2}(\Omega)}^{1-\theta} \, \|u\|_{L^{6}(\Omega)}^{\theta} , \quad \forall \theta\in [0,1] ;
\]
with $p_\theta$ defined by $\frac{1}{p_\theta} = \frac{1}{2}
(1-\theta) + \frac{1}{6} \theta$, we conclude that
\begin{equation}\label{ineq-diff}
  \|u - \tilde{u}\|_{L^{4}(\Omega)} \leq C \omega^2 \|{\delta c}^{-2}\|_{L^{2}(\Omega)} \|u\|_{L^{2}(\Omega)}^{3/4} \, \|u\|_{L^{6}(\Omega)}^{1/4} .
\end{equation}

Then we estimate the left-hand side of \eqref{break-diff-lp} using
inequalities \eqref{ineq-diff}, \eqref{ineq:1} and the
interpolation of $L^4$ space,
\begin{equation}\label{ineq-diff-4}
\begin{aligned}
& \hspace*{-1.5cm}
\left| \langle (\Lambda_{\omega^2(c^{-2} + {\delta c}^{-2})} - \Lambda_{\omega^2 c^{-2}})g \, , \, h \rangle -  \omega^2\int_\Omega {\delta c}^{-2} \, \tilde{u} v \,\dd x \right|
=  \,\, \left|\omega^2 \int_\Omega {\delta c}^{-2} \, (u - \tilde{u}) v \, \dd x \right|
\\
\leq & \,\,  \omega^2 \|{\delta c}^{-2}\|_{L^2(\Omega)} \|u - \tilde{u}\|_{L^4(\Omega)} \|v\|_{L^4(\Omega)}
\\[0.1cm]
\leq & \,\,  C \omega^4 \|{\delta c}^{-2}\|_{L^2(\Omega)}^2 \|u\|_{L^{2}(\Omega)}^{3/4} \|u\|_{L^{6}(\Omega)}^{1/4}
\|v\|_{L^{2}(\Omega)}^{3/4} \|v\|_{L^{6}(\Omega)}^{1/4}
\end{aligned}
\end{equation}
for some constant $C$. Next, by applying the Sobolev embedding,
$H^1(\Omega) \subset L^6(\Omega)$, and
Proposition~\ref{uniformconstants}, we arrive at the following
inequality,
\[
 \|u\|_{L^{2}(\Omega)}^{3/4} \|u\|_{L^{6}(\Omega)}^{1/4}
\leq C \|u\|_{L^{2}(\Omega)}^{3/4} \|u\|_{H^{1}(\Omega)}^{1/4}
\leq  C \|u\|_{H^{1/2}(\partial\Omega)}.
\]
This, with \eqref{ineq-diff-4} and the same procedure for $v$, leads
to the Fr\'echet differentiability of $F_\omega$ at $c^{-2}$.
\end{proof}

\medskip\medskip

\begin{lemma}\label{lemma:L-hat}
Let $c^{-2} \in L^2(\Omega)$ satisfy Assumption
\ref{Aprioribound}. Then, if (\ref{smallfr}) ((\ref{higherfr})) holds,
there exists a constant $\hat{\mathfrak{L}}_0$, which depends on
$(\Omega, B_2,\lambda_1,\omega^2_0)$
$((\Omega,B_1. B_2,\omega^2_1,\omega^2_2,\Sigma))$ such that
\begin{equation}\label{L-bd}
\| DF_\omega(c^{-2})\|_{\mathcal{L}(L^2(\Omega), \mathcal{L}(H^{1/2}(\partial\Omega), H^{-1/2}(\partial\Omega)))} \le \hat{\mathfrak{L}}_0 \omega^2 ;
\end{equation}
$\hat{\mathfrak{L}}_0$ blows up as
$|\omega_0^2-\frac{\lambda_1}{B_2}|\rightarrow 0$ $(\min
(|\omega_1^2-\frac{\lambda_n}{B_1}|,|\omega_2^2-\frac{\lambda_{n+1}}{B_2}|)
\rightarrow 0)$.
\end{lemma}

\begin{proof}
We start from Alessandrini's identity \eqref{Alessandrini-iden}. By
applying the H\"older inequality twice, we find that
\begin{multline*}
   | \langle (\Lambda_{\omega^2 c_1^{-2}}
     - \Lambda_{\omega^2 c_2^{-2}}) \, u_1 |_{\partial\Omega} \, , \,
                u_2 |_{\partial\Omega} \rangle |
   = \left| \int_{\Omega} \omega^2 (c_1^{-2} -c_2^{-2})
                          u_1 u_2 \, \dd x \right|
\\
  \le \omega^2 \| c_1^{-2} - c_2^{-2} \|_{L^{2}(\Omega)}
               \|u_1\|_{L^{4}(\Omega)} \|u_2\|_{L^{4}(\Omega)} .
\end{multline*}
Similarly as in the proof of Lemma~\ref{lemma:L}, by the interpolation
of $L^4$ space, the Sobolev embedding theorem
and Proposition~\ref{uniformconstants}, we obtain
\[
   \|u_i\|_{L^{4}(\Omega)} \le C \| u_i \|_{H^1(\Omega)}
        \le C \|u_i\|_{H^{1/2}(\partial \Omega)} , \quad i=1,2 ,
\]
where $C$ is a generic constant. Hence,
\[
  \|\Lambda_{\omega^2 c_1^{-2}} - \Lambda_{\omega^2 c_2^{-2}}\|_{\mathcal{L}(H^{1/2}(\partial\Omega), H^{-1/2}(\partial\Omega))} \le \hat{\mathfrak{L}}_0 \omega^2 \|c_1^{-2} -c_2^{-2}\|_{L^{2}(\Omega)},
\]
with $\hat{\mathfrak{L}}_0 = C \norm{u_1}_{L^{4}(\Omega)}$,
from which (\ref{L-bd}) follows.
\end{proof}

\medskip\medskip

\begin{lemma}\label{lemma:L}
Let $c_1^{-2}, c_2^{-2} \in L^2\Omega)$ satisfy Assumption
\ref{Aprioribound}, Then, if (\ref{smallfr}) ((\ref{higherfr})) holds
there exists a positive constant $\mathfrak{L}_0$, such that
\[
 \|DF_\omega(c_1^{-2}) - DF_\omega(c_2^{-2})\|_{\mathcal{L}(L^2(\Omega),\mathcal{L}(H^{1/2}(\partial\Omega), H^{-1/2}(\partial\Omega)))} \le \mathfrak{L}_0 \omega^4 \|c_1^{-2} - c_2^{-2}\|_{L^{2}(\Omega)},
 \]
where $\mathfrak{L}_0=\mathfrak{L}_0(\Omega, B_2,\lambda_1,\omega^2_0)$ $(\mathfrak{L}_0=\mathfrak{L}_0(\Omega,B_1. B_2,\omega^2_1,\omega^2_2,\Sigma))$ and $\mathfrak{L}_0$ blows up as $|\omega_0^2-\frac{\lambda_1}{B_2}|\rightarrow 0$
 $(\min (|\omega_1^2-\frac{\lambda_n}{B_1}|,|\omega_2^2-\frac{\lambda_{n+1}}{B_2}|) \rightarrow 0)$.
\end{lemma}

\begin{proof}
Let $g, h \in H^{1/2}(\Omega)$ and $u_i, v_i,\ i=1,2$, solve the
boundary value problems,
\[
\left\{
\begin{array}{rl}
(-\Delta - \omega^2 c_i^{-2})u_i = & 0, \quad x\in \Omega, \\
u_i = & h, \quad x\in \partial \Omega ,
\end{array}
\right.
\]
\[
\left\{
\begin{array}{rl}
(-\Delta - \omega^2 c_i^{-2})v_i = & 0, \quad x\in \Omega, \\
v_i = & g, \quad x\in \partial \Omega ,
\end{array}
\right.
\]
respectively. By using identity \eqref{Fre-DF} and applying the
H\"older inequality twice, we have
\begin{equation}
\begin{aligned}
   & \hspace*{-1.5cm}
 |\langle (DF_\omega(c_1^{-2})({\delta c}^{-2}) - DF_\omega(c_2^{-2})({\delta c}^{-2})) g \, , \, h\rangle|
\\
   = \,\, & \left| \omega^2 \int_{\Omega} {\delta c}^{-2} \, (u_1v_1 - u_2 v_2) \, \dd x \right|
\\
\le \,\, & \omega^2 \|{\delta c}^{-2}\|_{L^{2}(\Omega)}(\|u_1 - u_2\|_{L^{4}(\Omega)}\|v_1\|_{L^{4}(\Omega)} + \|u_2\|_{L^{4}(\Omega)}\|v_1 - v_2\|_{L^{4}(\Omega)}).
\end{aligned}
\end{equation}
We note that $u_1-u_2$ solves the equations
\[
\left\{
\begin{array}{rll}
(-\Delta - \omega^2 c_1^{-2})(u_1 - u_2) = & \omega^2 (c_1^{-2} - c_2^{-2}) \,  u_2 , & \quad x\in \Omega, \\
u_1 - u_2 = & 0, & \quad x\in \partial \Omega.
\end{array}
\right.
\]
Using an argument similar to the one in the proof of Lemma~\ref{Frechet-diff}, we derive that
\begin{equation}
\|u_1 - u_2\|_{L^{4}(\Omega)} \le C \omega^2 \|c_1^{-2} -c_2^{-2}\|_{L^{2}(\Omega)} \|u_2\|_{L^{4}(\Omega)} \le C \omega^2 \|c_1^{-2} -c_2^{-2}\|_{L^{2}(\Omega)} \|g\|_{H^{1/2}(\partial\Omega)}
\end{equation}
and, analogously,
\begin{equation}
\|v_1 - v_2\|_{L^{4}(\Omega)} \le C \omega^2 \|c_1^{-2} -c_2^{-2}\|_{L^{2}(\Omega)} \|v_2\|_{L^{4}(\Omega)} \le C \omega^2 \|c_1^{-2} -c_2^{-2}\|_{L^{2}(\Omega)} \|h\|_{H^{1/2}(\partial\Omega)}.
\end{equation}
Hence
\begin{multline}
|\langle (DF_\omega(c_1^{-2})({\delta c}^{-2}) - DF_\omega(c_2^{-2})({\delta c}^{-2})) g \, , \, h\rangle|
\\
\le C \omega^4 \|{\delta c}^{-2}\|_{L^{2}(\Omega)} \|c_1^{-2} -c_2^{-2}\|_{L^{2}(\Omega)} \|g\|_{H^{1/2}(\partial\Omega)} \|h\|_{H^{1/2}(\partial\Omega)},
\end{multline}
which gives that
\begin{equation}
\|DF_\omega(c_1^{-2}) - DF_\omega(c_2^{-2})\|_{\mathcal{L}(L^{2}(\Omega), \mathcal{L}(H^{1/2}(\partial\Omega), H^{-1/2}(\partial\Omega)))}
 \le C \omega^4 \|c_1^{-2} - c_2^{-2}\|_{L^{2}(\Omega)}.
\end{equation}
\end{proof}

\begin{remark}
In the above lemmas, we analyze the properties of the Fr\'echet
derivative. For simplicity, we use $L^2$. With
Assumption~\ref{Aprioribound}, we enforce that $c^{-2}$ belongs to
$L^\infty(\Omega)$. For $L^p$ optimization, $p > 3/2$,
Lemma~\ref{Frechet-diff}, \ref{lemma:L-hat} and \ref{lemma:L} can be
generalized with the $L^2$ norm replaced by the $L^p$ norm; the
constants $\mathfrak{L}_0$ and $\hat{\mathfrak{L}}_0$ will also depend
on $p$.
\end{remark}

\section{Stability of the inverse problem}\label{sec:stab}

Let $B_2, r_0, r_1, A, L, N$ be positive with $N \in
\mathbb{N},  N \geq 2$, $r_0 < 1$. In the sequel we will refer to these numbers as to
the a-priori data. To prove the results of this section we invoke the
following common assumptions

\medskip\medskip

\begin{assumption}\label{assumption_domain}
$\Omega \subset \mathbb{R}^n$ is a bounded domain such that
\[
   |x| \leq Ar_1, \quad\forall x\in \Omega.
 \]
Assume
\[
   \partial\Omega \text{  of Lipschitz class with  constants } r_1
   \text{ and } L.
\]
Consider a partitioning $\mathcal{D}_N$ of  $\Omega$  given by
\begin{equation}
   \mathcal{D}_N \triangleq \{\{D_1,D_2, \dots, D_N\} \mid
   \bigcup_{j=1}^N \overline{D}_j = \Omega \,\, , \,\, (D_j \cap D_{j'})^\circ =
   \emptyset, j\neq j' \}
\end{equation}
with
\[
   \{\partial D_j\}_{j=1}^{N}
   \text{  of Lipschitz class with  constants } r_0
   \text{ and } L.
\]

\end{assumption}

\medskip\medskip

\begin{assumption}\label{assumption_potential}
The function $c^{-2} \in W_N$, that is, satisfies
\[
B_1\leq c^{-2}\leq B_2,\quad \text{in }\Omega
\]
and is of the form
\[
   c^{-2}(x) = \sum_{j = 1}^N c_j \chi_{D_j}(x) ,
\]
where $c_j, j=1,\dots,N$ are unknown numbers and $D_j$ are known open
sets in $\mathbb{R}^3$.
\end{assumption}

\medskip\medskip

\begin{assumption}\label{wellpos}
\[
0<\omega^2\notin\Sigma_{c^{-2}} ,\quad \forall c^{-2} \in W_N .
\]
\end{assumption}

Under the above assumptions we can state the following result

\medskip\medskip

\begin{lemma}\label{propHs'}
Let assumptions \ref{assumption_domain} and \ref{assumption_potential}
hold and let $c^{-2}\in W_N$.  Then, for every $s'\in
(0,1/2)$, there exists a positive constant $C$ with $C=C(L,s')$ such
that
\begin{equation}\label{Hsbound}
\|c^{-2}\|_{H^{s'}(\Omega)}\leq  C(L,s')\frac{1}{r_0^{s'}}\|c^{-2}\|_{L^{2}(\Omega)} .
\end{equation}
\end{lemma}

\begin{proof}
The proof is based on the extension of a result of Magnanini and Papi in \cite{Magnanini1985} to the three dimensional setting.
In fact, following the argument in [MP], one has that
 \begin{equation}\label{bound1}
\|\chi_{D_j}\|^2_{H^{s'}(\Omega)}\leq \frac{16\pi}{(1-2s')(2s')^{1+2s'}}|D_j|^{1-2s'}|\partial D_j|^{2s'} .
\end{equation}
By Assumption \ref{assumption_potential},  recalling that the $\{D_j\}_{j=1}^N$  is a partition of disjoint sets of $\Omega$, we have that
\begin{equation}\label{Hsbound2}
\|c^{-2}\|^2_{H^{s'}(\Omega)}=\sum_{j=1}^Nc_j^2\|\chi_{D_j}\|^2_{H^{s'}(\Omega)}
\end{equation}
so that from (\ref{bound1})
\[
\|c^{-2}\|^2_{H^{s'}(\Omega)}\leq  C(s')\sum_{j=1}^Nc_j^2|D_j|\left(\frac{|\partial D_j|}{|D_j|}\right)^{2s'}\leq \frac{C(L, s')}{r_0^{2s'}}\|c^{-2}\|^2_{L^{2}(\Omega)} .
\]
\end{proof}

We have the following stability result

\medskip\medskip

\begin{proposition}\label{prop:Stab}
Assume  (\ref{assumption_domain}) and let $c^{-1}_1 ,c^{-1}_2\in \tilde{W}_N$  and  $\omega^2$ satisfies (\ref{wellpos}) . Then, there exists a positive constant  $K$ depending on $A, r_1, L$ such that,
\begin{equation}\label{Lip_stab-1}
   \| c_1^{-2} - c_2^{-2}\|_{L^{2}(\Omega)} \le
       \frac{1}{\omega^2} \, e^{K(1+\omega^2B_2) (|\Omega|/r^3_0)^{{\frac{4}{7}}}} \, \, \|\Lambda_1
   - \Lambda_2\|_{\mathcal{L}(H^{1/2}(\partial\Omega),
                H^{-1/2}(\partial\Omega))},
\end{equation}
where  $\Lambda_k = \Lambda_{\omega^2 c_k^{-2}}$ for
$k=1,2$.
\end{proposition}

\begin{proof}
To prove our stability estimate we follow the original idea of
Alessandrini of using CGO solutions but we use slightly different ones
than those introduced in \cite{Sylvester1987} and in
\cite{Alessandrini1988} to obtain better constants in the stability
estimates as proposed by \cite{Salo}.  We also use the estimates
proposed in \cite{Salo} (see Theorem 3.8) and due to \cite{Haehner1996}
concerning the case of bounded potentials.

In fact, by Theorem 3.9 of \cite{Salo}, since $c^{-2}\in
L^{\infty}(\Omega)$, $\|c^{-2}\|_{L^{\infty}(\Omega)}\leq B_2$, then
there exists a positive constant $C$ such that for any
$\zeta\in\mathbb{C}^3$ satisfying $\zeta\cdot\zeta=0$ and $|\zeta|\geq
C$ the equation
\[
-\Delta u-\omega^2 c^{-2}u=0
\]
has a solution of the form
\[
u(x)=e^{ix\cdot\zeta}(1+R(x))
\]
where $r\in H^1(\Omega)$ satisfies
\[
\|R\|_{L^2(\Omega)}\leq \frac{C}{|\zeta|},
\]
\[
\|\nabla R\|_{L^2(\Omega)}\leq C.
\]


Let $\xi\in \mathbb{R}^3$ and let $\tilde\omega_1$ and $\tilde\omega_2$ be unit vectors of $\mathbb{R}^3$ such that $\{\tilde\omega_1,\tilde\omega_2,\xi\}$ is an orthogonal set of vectors of $\mathbb{R}^3$ . Let $s$ be a positive parameter to be chosen later and set for $k=1,2$,
\begin{equation}
\zeta_k=\left\{\begin{array}{ccc}
                  (-1)^{k-1}\frac{s}{\sqrt 2}(\sqrt{(1-\frac{|\xi|^2}{2s^2})}\tilde\omega_1+(-1)^{k-1}\frac{1}{\sqrt{2}s}\xi+\ii\tilde\omega_2) &\mbox{for}& \frac{|\xi|}{\sqrt 2s}<1 , \\
                   (-1)^{k-1}\frac{s}{\sqrt 2}((-1)^{k-1}\frac{1}{\sqrt{2}s}\xi+\ii(\sqrt{(\frac{|\xi|^2}{2s^2}-1)}\tilde\omega_1+\tilde\omega_2))& \mbox{for}& \frac{|\xi|}{\sqrt 2s}\geq 1 .
                 \end{array}
\right.
\end{equation}
Then an straightforward computation gives
\[\zeta_k\cdot\zeta_k=0\]
for $k=1,2$ and
\[\zeta_1+\zeta_2=\xi.\]
Furthermore, for $k=1,2$,
 \begin{equation}
|\zeta_k|=\left\{\begin{array}{ccc}
             s &\mbox{for}& \frac{|\xi|}{\sqrt 2s}<1, \\
                  \frac{|\xi|}{\sqrt2}& \mbox{for}& \frac{|\xi|}{\sqrt 2s}\geq 1.
                 \end{array}
\right.
\end{equation}
Hence,
\begin{equation}\label{zeta}
|\zeta_k|=\max\{s,\frac{|\xi|}{\sqrt 2}\}.
\end{equation}
Then, by Theorem 3.9 of \cite{Salo}, for $|\zeta_1|,|\zeta_2|\geq C_1=\max\{C_0\omega^2B_2,1\}$, with $C_0=C_0(A,r_1)$, there exist $u_1,u_2$,  solutions to $-\Delta u_k- \omega^2 c_k^{-2} u_k=0$  for $k=1,2$ respectively, of the form
\begin{equation}\label{cgosol}
u_1(x)=e^{\ii x\cdot\zeta_1}(1+R_1(x)),\quad u_2(x)=e^{\ii x\cdot\zeta_2}(1+R_2(x))
\end{equation}
 with
\begin{equation}\label{remindercgo}
\|R_k\|_{L^2(\Omega)}\leq \frac{C_0\sqrt{|\Omega|}}{s}\omega^2B_2
\end{equation}
and
\begin{equation}\label{remindergradcgo}
\|\nabla R_k\|_{L^2(\Omega)}\leq C_0\sqrt{|\Omega|}\omega^2B_2
\end{equation}
for $k=1,2$.

Consider Alessandrini's identity
\[
\int_{\Omega} \omega^2(c_1^{-2} - c_2^{-2})u_1 u_2 \dd x = \langle(\Lambda_1-\Lambda_2)u_1|_{\partial\Omega},u_2|_{\partial\Omega}\rangle,
\]
where  $u_k\in H^1(\Omega)$ is any solution of $-\Delta u_k-\omega^2 c^{-2}_k u_k=0$ for $k=1,2$. Inserting the solutions (\ref{cgosol})  in Alessandrini's identity we derive
\begin{multline}
\left|\int_{\Omega} \omega^2(c^{-2}_1-c^{-2}_2)e^{\ii\xi\cdot x}\dd x\right| \leq \|\Lambda_1-\Lambda_2\|\|u_1\|_{H^{1/2}(\partial\Omega)}\|u_2\|_{H^{1/2}(\partial\Omega)}
\\
+\left|\int_{\Omega}\omega^2(c^{-2}_1-c^{-2}_2)e^{\ii\xi\cdot x}(R_1+R_2+R_1R_2)\dd x\right|\\
\leq \|\Lambda_1-\Lambda_2\|\|u_1\|_{H^{1}(\Omega)}\|u_2\|_{H^{1}(\Omega)}+E(\|R_1\|_{L^2(\Omega)}+\|R_2\|_{L^2(\Omega)}+\|R_1\|_{L^2(\Omega)}\|R_2\|_{L^2(\Omega)}).
\end{multline}
where  $E:=\|\omega^2(c^{-2}_1-c^{-2}_2)\|_{L^{2}(\Omega)}$
By (\ref{remindercgo}), (\ref{remindergradcgo}), (\ref{zeta}) and since $\Omega\subset B_{2R}(0)$  we have
\[
\|u_k\|_{H^{1}(\Omega)}\leq C\sqrt{|\Omega|}(s+|\xi|) e^{Ar_1(s+|\xi|)},\quad k=1,2.
\]
Let $s\geq C_2$  so that $s+|\xi|\leq e^{Ar_1(s+|\xi|)}$. Then, for $s\geq C_3=\max(C_1,C_2)$, using (\ref{remindercgo}) and (\ref{remindergradcgo}) we get
\begin{equation}\label{bound2}
   |\omega^2(c^{-2}_1-c^{-2}_2)\,\hat{ }\,(\xi)| \leq
   C\sqrt{|\Omega|} \left( e^{4Ar_1(s+|\xi|)}\|\Lambda_1-\Lambda_2\|+\frac{\omega^2B_2E}{s} \right)
\end{equation}
where the $\omega^2 c^{-2}_k$'s  have been extended to all $\mathbb{R}^3$ by zero. Hence, we get
\begin{equation}\label{bound3}
\|\omega^2 (c^{-2}_1-c^{-2}_2)\,\hat{ } \,\|^2_{L^2(\mathbb{R}^3)}\leq C |\Omega|\rho^3\left(e^{8Ar_1(s+\rho)}\|\Lambda_1-\Lambda_2\|^2+\frac{\omega^4B^2_2E^2}{s^2}\right)+\int_{|\xi|\geq \rho}|\omega^2 (c^{-2}_1-c^{-2}_2)\,\hat{ }\,(\xi)|^2\,\dd\xi
\end{equation}
where $C=C(A,r_1)$.
By (\ref{Hsbound}) and (\ref{Hsbound2}) we have that
\[
\|\omega^2(c^{-2}_1-c^{-2}_2)\|^2_{H^{s'}(\Omega)} \leq \frac{C}{r_0^{2s'}}E^2,
\]
where $C$ depends on $L, s'$ and hence
\begin{eqnarray*}
\rho^{2s'}\int_{|\xi|\geq \rho}|\omega^2 (c^{-2}_1-c^{-2}_2)\,\hat{ }\,(\xi)|^2\,\dd\xi&\leq &\int_{|\xi|\geq \rho}|\xi|^{2s'}|\omega^2 (c^{-2}_1-c^{-2}_2)\,\hat{ }\,(\xi)|^2\,\dd\xi\\
&\leq &\int_{\mathbb{R}^3}(1+|\xi|^{2})^{s'}|\omega^2 (c^{-2}_1-c^{-2}_2)\,\hat{ }\,(\xi)|^2\,\dd\xi\leq \frac{C}{r_0^{2s'}}E^2.
\end{eqnarray*}
Hence,  we get
\[
\int_{|\xi|\geq \rho}|\omega^2 (c^{-2}_1-c^{-2}_2)\,\hat{ }\,(\xi)|^2\,\dd\xi\leq \frac{CE^2}{r^{2s'}_0\rho^{2s'}}
\]
for any $s'\in (0,1/2)$. Inserting last bound in (\ref{bound3})  we derive
\[
\|\omega^2 (c^{-2}_1-c^{-2}_2)\,\hat{ } \,\|^2_{L^2(\mathbb{R}^3)}\leq C\left(\rho^3|\Omega|\left( e^{8Ar_1(s+\rho)}\|\Lambda_1-\Lambda_2\|^2+\frac{\omega^4B^2_2E^2}{s^2}\right)+\frac{E^2}{r_0^{2s'}\rho^{2s'}}\right).
\]
where $C=C(L,s')$.
Pick up
\[\sqrt[3]{|\Omega|}\rho=\left(\frac{|\Omega|}{r_0^3}\right)^{\frac{2s'}{3(3+2s')}}\left(\frac{1}{\alpha}\right)^{\frac{1}{3+2s'}}s^{\frac{2}{3+2s'}}
\]
with $\alpha=\max \{1,\omega^4B_2^2\}$. Then, by Assumption~ \ref{assumption_domain} and observing that we might assume without loss of generality that $\frac{|\Omega|}{r_0^3}>1$, we get
\[
\|\omega^2 (c^{-2}_1-c^{-2}_2)\|^2_{L^{2}(\Omega)}\leq CE^2\left(\frac{|\Omega|}{r_0^3}\right)^{\frac{2s'}{3+2s'}}\left(e^{C_4(\frac{|\Omega|}{r^3_0})^{\frac{2s'}{3(3+2s')}}s}\left(\frac{\|\Lambda_1-\Lambda_2\|}{E}\right)^2+\left(\frac{\alpha}{s^2}\right)^{\frac{2s'}{3+2s'}}\right)
\]
for $s\geq C_3$ and where $C $ depends on $s', L, A,r_1$ and $C_4$ depends on $L,A,r_1$. Let us choose
\[
s=\frac{1}{C_4(\frac{|\Omega|}{r^3_0})^{\frac{2s'}{3(3+2s')}}}\left|\log \frac{\|\Lambda_1-\Lambda_2\|}{E}\right|
\]
where we have assumed that
\[
\frac{\|\Lambda_1-\Lambda_2\|}{E} < c:=e^{-\bar C\max\{1,\omega^2B_2\}(\frac{|\Omega|}{r^3_0})^{\frac{2s'}{3(3+2s')}}}
\]
with $\bar C=\bar C(R)$ so that $s\geq C_3$.  Under this assumption,
\begin{equation}\label{logstab}
\|\omega^2 (c^{-2}_1-c^{-2}_2)\|_{L^{2}(\Omega)}\leq C(\sqrt\alpha)^{\frac{2s'}{3+2s'}}\left(\frac{|\Omega|}{r^3_0}\right)^{\frac{2s'}{3+2s'}\frac{9+10s'}{6(3+2s')}}E\left(\left|\log \frac{\|\Lambda_1-\Lambda_2\|}{E}\right|^{-\frac{2s'}{3+2s'}}\right)
\end{equation}
 where $C=C( L,s',A,r_1)$ and we can rewrite last inequality in the form
\begin{equation}\label{logstab1}
E\leq C(1+\omega^2B_2)^{\frac{2s'}{3+2s'}}\left(\frac{|\Omega|}{r^3_0}\right)^{\frac{2s'}{3+2s'}\frac{9+10s'}{6(3+2s')}}E\left(\left|\log \frac{\|\Lambda_1-\Lambda_2\|}{E}\right|^{-\frac{2s'}{3+2s'}}\right)
\end{equation}
which gives
\begin{equation}\label{lip1}
E\leq e^{C(1+\omega^2B_2)(\frac{|\Omega|}{r^3_0})^{\frac{9+10s'}{6(3+2s')}}}\|\Lambda_1-\Lambda_2\|
\end{equation}
where $C=C( L,s',A,r_1)$.
On the other hand if
\[
\frac{\|\Lambda_1-\Lambda_2\|}{E}\geq c,
\]
then
\begin{equation}\label{lipstab}
\|\omega^2 (c^{-2}_1-c^{-2}_2)\|_{L^{2}(\Omega)}\leq c^{-1}\|\Lambda_1-\Lambda_2\|\leq e^{\bar C(1+\omega^2B_2)\left(\frac{|\Omega|}{r^3_0}\right)^{\frac{1}{3(3+2s')}}}\|\Lambda_1-\Lambda_2\|
\end{equation}
Hence, from (\ref{lip1}) and (\ref{lipstab}) and recalling that $s\in (0,\frac{1}{2})$, we have that
\begin{equation}\label{lip2}
E\leq e^{C(1+\omega^2B_2)(\frac{|\Omega|}{r^3_0})^{\frac{9+10s'}{6(3+2s')}}}\|\Lambda_1-\Lambda_2\|
\end{equation}
Choosing $s'=\frac{1}{4}$, we derive
\[
\|c^{-2}_1-c^{-2}_2\|_{L^{2}(\Omega)}\leq \frac{1}{\omega^2}e^{K(1+\omega^2B_2)(|\Omega|/r^3_0)^{\frac{4}{7}}}\|\Lambda_1-\Lambda_2\|
\]
where $K=K(L,A,r_1,s')$ and the claim follows.
\end{proof}

\medskip\medskip

\begin{remark} \label{rem:cubpart}
In \cite{Beretta2012} the following lower bound of the stability constant has been obtained in the case of a uniform polyhedral partition $\mathcal{D}_N$
\[
C_N\geq \frac{1}{4}e^{K_1N^{\frac{1}{5}}}
\]
Choose a uniform cubical partition $\mathcal{D}_N$ of $\Omega$ of mesh size $r_0$. Then,
\[
|\Omega|=Nr_0^3
\]
and estimate (\ref{Lip_stab-1}) of Proposition \ref{prop:Stab} gives
\[
C_N= \frac{1}{\omega^2}e^{K(1+\omega^2B_2)N^{\frac{4}{7}}},
\]
which proves optimality of the bound on the Lipschitz constant with respect to $N$ when the global DtN map is known.
In \cite{Beretta2012} a Lipschitz stability estimate has been derived
in terms of the local DtN map using singular solutions.
This type of solutions allows to recover the unknown piecewise
constant wavespeeds by determining it on the outer boundary of the
domain and then, by propagating the singularity inside the domain, to
recover step by step the wavespeed on the interface of all subdomains
of the partition. This iterative procedure does not lead to sharp
bounds of the Lipschitz constant appearing in the stability
estimate. It would be interesting if one can get a better bound of the
Lipschitz constant using oscillating solutions.
\end{remark}

\medskip\medskip

\begin{remark}
Observe that the result obtained in Proposition \ref{prop:Stab}
extends to the case
$$
c^{-2}(x)=\sum_{i=1}^Nc_j\psi_j(x)
$$
with  $\psi_j$  with support in $D_j$ for each $j$ and satisfying
\[
r_0^{s'}\|\psi_j\|_{H^{s'}(D_j)}\leq M,\quad j=1,\cdots, N
\]
for $s'\in (0,\frac{1}{2})$.
Then
\begin{equation}
\|c^{-2}\|^2_{H^{s'}(\Omega)}=\sum_{j=1}^Nc_j^2\|\psi_j\|^2_{H^{s'}(D_j)}\leq  \frac{M}{r_0^{2s'}}\|c^{-2}\|^2_{L^2(\Omega)}
\end{equation}
and we might apply the same arguments of Proposition \ref{prop:Stab}
to derive estimate (\ref{Lip_stab-1}).
\end{remark}

\medskip\medskip

\begin{remark}
Observing that
\[
\|c^{-2}_1-c^{-2}_2\|_{L^{\infty}(\Omega)}\leq \frac{C}{r_0^{3/2}}\|c^{-2}_1-c^{-2}_2\|_{L^{2}(\Omega)},
\]
we immediately get the following stability estimate in the
$L^{\infty}$ norm
\[
\|c^{-2}_1-c^{-2}_2\|_{L^{\infty}(\Omega)}\leq \frac{C}{\omega^2}e^{K(1+\omega^2B_2)(|\Omega|/r^3_0)^{\frac{4}{7}}}\|\Lambda_1-\Lambda_2\|
\]
with $C=C(L,r_0)$.
\end{remark}

\section{Multi-level projected steepest descent iteration in
         $L^{2}(\Omega)$}

Thus far, we have analyzed the data operator, which is contained in
the Banach space, $Y$ say, of bilinear form on
$H^{1/2}(\partial\Omega) \times H^{1/2}(\partial\Omega)$, and the
stability of the inverse problem, and quantitatively estimated the
relevant constants. Here, we consider the iterative reconstruction in
$X := L^2(\Omega)$. We apply a projected
steepest descent iteration \cite{Hoop2012a} in a multi-level setting.

In the following lemma, we summarize some basic notions associated
with iterative methods in $L^2(\Omega)$. For a detailed introduction
to the Bregman distance, the duality mapping and the geometry of
Banach spaces, we refer to \cite{Cioranescu1990,Schopfer2006}.

\medskip\medskip

\begin{lemma}
\begin{enumerate}[(a)]
\item
The normalized duality mapping, $J_2$, is the identity mapping.
\item
The Bregman distance, $\Delta_2(f,\tilde{f})= \frac{1}{2} \|f-\tilde{f}\|_{L^{2}(\Omega)}^{2} $.
\item
Given a closed convex set $Z \subset L^2(\Omega)$, the projection of an element $f \in L^2(\Omega)$ into $Z$ given by
\begin{equation}\label{projection}
   P_Z(f)
       = \argmin_{\tilde{f}\in L^2(\Omega)} \Delta_2(f,\tilde{f}) ,
\end{equation}
is non-expansive.
\item Assume that $\{D_j\}_{j=1}^N$ is a domain partitioning of $\Omega$ as
in Section~\ref{sec:stab}. Let $Z$ be defined by $ Z =
\operatorname{span}\{ \chi_{D_1},\dots, \chi_{D_N}\}$. We
have that
\[
   P_Z(f) = \sum_{j=1}^N g_j\, \chi_{D_j} \quad \mbox{with} \quad
   g_j = \frac{1}{|D_j|} \int_{D_j} f(x) \, \dd x .
\]
\end{enumerate}
\end{lemma}

\medskip\medskip

The assumptions concerning the Fr\'echet derivative, $DF$, of the
direct operator, $F :\ X \to Y$, associated with the projected
steepest descent iteration can be summarized as \cite{Hoop2012a}

\begin{enumerate}[(a)]
\item (uniformly bounded)
\[
\|DF\| \leq \hat{\mathfrak{L}}
\]
\item (Lipschitz continuous)
\[
\|DF(c_1^{-2}) - DF(c_2^{-2})\|
     \leq \mathfrak{L} \|c_1^{-2} - c_2^{-2}\| ,\quad
     \forall c_1^{-2}, c_2^{-2} \in X
\]
\item (conditionally stable)
\[
\|c_1^{-2} - c_2^{-2}\|
     \leq \mathfrak{C} \|F(c_1^{-2}) - F(c_2^{-2})\| ,\quad
     \forall c_1^{-2}, c_2^{-2} \in Z \subset X .
\]
\end{enumerate}

In the previous sections, we examined these conditions and estimated
the behaviors of the constants with respect to the frequency,
$\omega$, and the number, $N$, of subdomains in the domain
partitioning. With Lemma~\ref{lemma:L-hat}, Lemma~\ref{lemma:L} and
Proposition~\ref{prop:Stab}, we have
\begin{equation}\label{cns}
\begin{aligned}
  \hat{\mathfrak{L}} =& \hat{\mathfrak{L}}_0 \omega^2 ,
\\
         \mathfrak{L}=& \mathfrak{L}_0 \omega^4 ,
\\
         \mathfrak{C} =& \omega^{-2}
            e^{K (1 + \omega^2 B_2) (|\Omega| / r_0^3)^{\frac{4}{7}}}
\end{aligned}
\end{equation}
while assuming that $\omega^2$ satisfies the conditions in
Proposition~\ref{uniformconstants}. We elucidated in
Remark~\ref{rem:cubpart} that $|\Omega| / r_0^3$ is effectively equal
to $N$ in the case of a uniform domain partitioning.

We now introduce the framework of discrete approximations of the
unique solution. Let $c_{\dagger}$ denote the unique solution of the
inverse boundary value problem and $y_\omega$ the data, that is, $y_\omega =
F_{\omega}(c_{\dagger}^{-2})$.

\medskip\medskip

\begin{definition}\label{def:aprxerr}
For a given domain partitioning $\mathcal{D}_N$, the approximation
error, $\eta_{\omega, \mathcal{D}_N}$, is given by
\[
\begin{array}{cll}
   \eta_{\omega,\mathcal{D}_N} :\
   \mathcal{L}(H^{1/2}(\partial \Omega),
        H^{-1/2}(\partial \Omega)) &\rightarrow \,\,& [0,\infty)
\\
   \,\, y_{\omega} & \mapsto \,\,& \operatorname{dist}(y_{\omega},
            F_{\omega}(W_N) .
\end{array}
\]
\end{definition}

\medskip\medskip

We consider the error of the discrete approximation and assume that it
is bounded by some positive function $\varphi$ which is monotonically
decreasing in $N$:
\begin{equation}\label{phidef}
  \operatorname{dist}( c_\dagger^{-2},W_N ) \leq \varphi(N) .
\end{equation}
Lemma~\ref{lemma:L-hat} implies that
\begin{equation}\label{eta-behavior}
   \eta_{\omega,\mathcal{D}_N}(y)
            \leq \hat{\mathfrak{L}}_0 \omega^2 \varphi(N) .
\end{equation}
Indeed the error decreases as the frequency decreases.

The following projected steepest descent iteration is taken from
\cite{Hoop2012a}. We let $Z = W_N$.

\medskip\medskip

\begin{algorithm}\label{algo:1}
For $c_k^{-2}$, $k=0,1,2,\ldots$, we let
\begin{equation}
\label{abbreviation}
     R_k = F_{\omega,\mathcal{D}_N}(c_k^{-2}) - y_\omega \, ,\quad
     T_k = DF_{\omega,\mathcal{D}_N}(c_k^{-2})^*j_2(R_k) \, ,\quad
     r_k = \|R_k\| \, ,\quad t_k = \|T_k\| \; ,
\end{equation}
where $j_2$ stands for a single-valued selection of the normalized
duality mapping from the data space $Y$ to its dual. Moreover, we
define
\begin{equation}\label{abbreviation_2a}
\begin{aligned}
   \tilde{\mathfrak{C}} \, & := \mathfrak{L}\mathfrak{C}^2\,,
\\
   \rho & := \frac{1}{2}
(2\tilde{\mathfrak{C}}\hat{\mathfrak{L}})^{-2} \left(1+\sqrt{1-8\tilde{\mathfrak{C}}\eta} - 4\eta\tilde{\mathfrak{C}}\right)^2 .
\end{aligned}
\end{equation}
For $k=0,1,\ldots$ we set
\begin{equation}
\label{abbreviation_2}
\begin{aligned}
     u_k & := -\tilde{\mathfrak{C}} r_k^2
       + (1 - 2\tilde{\mathfrak{C}}\eta)r_k -\eta -\tilde{\mathfrak{C}} \eta^2 \,,\\
     v_k & := t_k^{-2} u_k r_k^{2}(r_k - \eta) - \frac{1}{2} t_k^{-2} u_k^2 r_k^{2} \,,\\
     w_k & :=\mathfrak{L} t_k^{-2}u_k  r_k^{2} \,,
\end{aligned}
\end{equation}
and
\begin{equation}
   \mu_k := t_k^{-2} u_k r_k \; .
\end{equation}
The iteration is given by
\begin{enumerate}[$(S1)$]
\setcounter{enumi}{-1}
\item Choose a starting point $c_0^{-2} \in Z$ such that
\begin{equation}\label{converge-radius0}
   \Delta_2(c_0^{-2},z^\dag) < \rho ,
\end{equation}
where $z^\dag \in Z$ satisfies
\[
   \operatorname{dist}(y_\omega,F_{\omega,\mathcal{D}_N}(z^\dag))
                  = \eta_{\omega,\mathcal{D}_N}(y_\omega) .
\]
\item Compute the new iterate via
\begin{equation}\label{Steepest descent}
\begin{array}{rl}
   \tilde{c}_{k+1}^{-2}  = &c_{k}^{-2} - \mu_k T_k
\\[0.2cm]
   c_{k+1}^{-2}  = & \mathcal{P}_Z(\tilde{c}_{k+1}^{-2}) .
\end{array}
\end{equation}
Set $k \leftarrow k + 1$ and repeat this step.
\end{enumerate}
\end{algorithm}

\medskip\medskip

Due to the projection $\mathcal{P}_Z$ applied, all iterates belong to
the `stable subset' $Z$, which in general can only offer an
approximation to the unique solution. The dimension, $N$, of $Z$
should be low to ensure a large radius of convergence, $\rho$
(cf.~(\ref{abbreviation_2a})). In \cite{Hoop2012a} we introduced a
multi-level approach to enable a gradual refinement of the domain
partitioning. Let $n$ denote the level index. As $n$ increases, the
number of subdomains, $N_n$, grows while the approximation error
decreases. We introduce a hierarchy of domain partitionings,
$\mathcal{D}_{N_n}$, and corresponding
\begin{equation}\label{stab-subset}
  Z_n = W_{N_n} .
\end{equation}
Given a frequency $\omega$, to each level, we assign a domain partitioning. We identify
$\mathcal{D}_{N_n}$ with $\mathcal{D}_n$,
$F_{\omega,\mathcal{D}_{N_n}}$ with $F_n$ and
$\eta_{\omega,\mathcal{D}_{N_n}}$ with $\eta_n$. For the given frequency $\omega$, the (attainable) data $y_{\omega}$ is denoted by $y$. We identify
$\hat{\mathfrak{L}}_n$, $\mathfrak{L}_n$ and $\mathfrak{C}_n$ with the
expressions in \eqref{cns} replacing $N$ by $N_n$. Similarly, we identify the radius $\rho_n$
with the expression for $\rho$ in \eqref{abbreviation_2a} replacing
$\hat{\mathfrak{L}}$, $\mathfrak{L}$, $\mathfrak{C}$ and $\eta$ by
$\hat{\mathfrak{L}}_n$, $\mathfrak{L}_n$, $\mathfrak{C}_n$ and $\eta_n$,
respectively. For simplicity of notation, we omit the subscript in the
operator norm. In the following algorithm, $c^{-2}_{n,k}$ denotes the
$k^{\mathrm{th}}$ iterate at level $n$.

\medskip\medskip

\begin{algorithm}\label{algo:prog}
\begin{enumerate}[$(S1)$]
\setcounter{enumi}{-1}
\item Set $n=0$. Use  $c^{-2}_{0,0}$ as the starting point, with
  $\Delta_2(c_{0,0}^{-2},c_{\dag,0}^{-2}) < \rho_0$ where
  $c_{\dag,0}^{-2}$ is the best $\mathcal{D}_0$ approximation to
  $c_{\dag}^{-2}$.
\\[0.1cm]
\item Iteration. Use $F_n$ and $Z_n$ as the modelling operator and the
  convex subset to run Algorithm~\ref{algo:1} with the discrepancy
  criterion given by
  \begin{equation}\label{discrepancy-principle-prog}
     K_n = \min \{ k \in\mathbb{N}  \mid
     \| F_n(c_{n,k}^{-2}) - y_n \| \le (3+\varepsilon)\eta_n \} ,
  \end{equation}
  where $\varepsilon > 0$ is a given small tolerance constant. That
  is, this algorithm stops at $k = K_n$.
\\[0.1cm]
\item Set $c^{-2}_{n+1,0} = c^{-2}_{n, K_n}$, refine the domain
  partitioning to $\mathcal{D}_{n+1}$ such that the corresponding constants
  $\hat{\mathfrak{L}}_{n+1}$, $\mathfrak{L}_{n+1}$,
  $\mathfrak{C}_{n+1}$ and the approximation error $\eta_{n+1}$
  satisfy the inequalities
  \begin{equation}\label{level-cont}
  \left\{\begin{array}{rll}
  8 \tilde{\mathfrak{C}}_{n+1}\eta_{n+1} < & 1 ,
\\
  (3 + \varepsilon) \eta_n +\eta_{n+1} \leq &
  2^{-5/2}  (\hat{\mathfrak{L}}_{n+1} \mathfrak{C}_{n+1}
               \tilde{\mathfrak{C}}_{n+1})^{-1} .
  \end{array}\right.
  \end{equation}
  Set $n = n+1$ and go to step $(S1)$.
\end{enumerate}
\end{algorithm}

\medskip\medskip

\begin{remark}\label{rem:mlc}
Because
\[
   1 - 8 \tilde{\mathfrak{C}}_{n+1}\eta_{n+1} > 0
\]
implies that
\[
   1 + \sqrt{1 - 8 \tilde{\mathfrak{C}}_{n+1} \eta_{n+1}}
              - 4 \tilde{\mathfrak{C}}_{n+1} \eta_{n+1} > 1/2 ,
\]
the inequalities in \eqref{level-cont} imply the original multi-level
criterion \cite{Hoop2012a}
\begin{equation}\label{cong-cri}
   (3 + \varepsilon) \eta_n <
   2^{-1/2} (\hat{\mathfrak{L}}_{n+1} \mathfrak{C}_{n+1})^{-1}
   \left(\frac{1 + \sqrt{1 -
         8 \tilde{\mathfrak{C}}_{n+1} \eta_{n+1}}}{
    2 \tilde{\mathfrak{C}}_{n+1}} - 2\eta_{n+1}\right) - \eta_{n+1} .
\end{equation}
\end{remark}

\medskip\medskip

This algorithm is adapted to the inverse boundary value problem for
the Helmholtz equation as follows. The multi-level conditions in
(\ref{level-cont}) are replaced by
\begin{equation}\label{eq:om1}
  \varphi(N_{n+1}) -
  8^{-1} \omega^{-2} (\mathfrak{L}_0 \hat{\mathfrak{L}}_0)^{-1}
    e^{-2K (1 + \omega^2 B_2) N_{n+1}^{\frac{4}{7}}} < 0 ,
\end{equation}
and
\begin{equation}\label{eq:om2}
   (3 + \varepsilon) \varphi(N_n) + \varphi(N_{n+1}) -
   2^{-5/2} \omega^{-2}(\hat{\mathfrak{L}}_0^2 \mathfrak{L}_0)^{-1}
     e^{-3K (1 + \omega^2 B_2) N_{n+1}^{\frac{4}{7}}} \le 0 .
\end{equation}
We now show that (\ref{eq:om1})-(\ref{eq:om2}) are the proper
multi-level conditions. These \textit{provide an upper bound for the
  increment in the number of subdomains from level to level}. One can
no longer increase $N$ when $N_{\mathrm{max}}$ is reached, which
follows from setting $N_{n+1} = N_n = N_{\mathrm{max}}$ in
(\ref{eq:om2}):
\begin{equation}
   (4 + \varepsilon) \varphi(N_{\mathrm{max}}) - 2^{-5/2}
   \omega^{-2}(\hat{\mathfrak{L}}_0^2 \mathfrak{L}_0)^{-1}
      e^{-3K (1 + \omega^2 B_2) N_{\mathrm{max}}^{\frac{4}{7}}} = 0 .
\end{equation}
We note that $\varphi$ signifies a compression rate while
approximating the unique solution with piecewise constant functions
corresponding with the given hierarchy of domain partitionings. If
$\varphi$ decays fast enough, such a (finite) $N_{\mathrm{max}}$ does
not exist.

Indeed, let $c_{\dagger,n}$ be the $\mathcal{D}_n$-best approximation
to $c_{\dagger}$ which gives the approximation error $\eta_n =
\eta_{\omega, \mathcal{D}_{N_n}}(y)$ in
Definition~\ref{def:aprxerr}. We have
\[
   \| c_{n, K_n}^{-2} - c_{\dagger}^{-2} \| \leq
   \| c_{n, K_n}^{-2} - c_{\dagger, n}^{-2} \|
         + \| c_{\dagger, n}^{-2} - c_{\dagger}^{-2} \| .
\]
For the first term, noting that both $c_{n, K_n}^{-2}$ and
$c_{\dagger, n}^{-2}$ belong to the `stable subset', we can apply the
stability estimate, Proposition~\ref{prop:Stab}, to arrive at
\begin{equation}\label{err-help1}
   \|c_{n,K_n}^{-2} - c_{\dagger,n}^{-2}\|
      \leq \omega^{-2} e^{K (1 + \omega^2 B_2) N_n^{\frac{4}{7}}}
                 \|F_n(c_{n, K_n}^{-2}) - F_n(c_{\dagger, n}^{-2})\| .
\end{equation}
Here, we have replaced $|\Omega| / r_0^3$ by the number, $N$, of
subdomains in the partitioning following Remark~\ref{rem:cubpart}.

By the discrepancy criterion \eqref{discrepancy-principle-prog}, we
know that
\[
   \| F_n(c_{n,K_n}^{-2}) - y_n \| \le (3+\varepsilon) \eta_n .
\]
Hence,
\[
   \| F_n(c_{n, K_n}^{-2}) - F_n(c_{\dagger, n}^{-2}) \| \le
   \| F_n(c_{n, K_n}^{-2}) - y_n \|
             + \| F_n(c_{\dagger, n}^{-2}) - y \|
   \le (4+\varepsilon) \eta_n .
\]
Substituting the above inequality into \eqref{err-help1}, we obtain
\[
   \| c_{n, K_n}^{-2} - c_{\dagger, n}^{-2} \|
   \le (4+\varepsilon) \omega^{-2}
                 e^{K (1 + \omega^2 B_2) N_n^{\frac{4}{7}}} \eta_n
\]
so that
\begin{equation}\label{errest-n}
   \| c_{n, K_n}^{-2} - c_{\dagger}^{-2} \| \leq
   (4+\varepsilon) \omega^{-2}
         e^{K (1 + \omega^2 B_2) N_n^{\frac{4}{7}}} \eta_n
                   + \| c_{\dagger, n}^{-2} - c_{\dagger}^{-2} \| .
\end{equation}

We use the estimates of $\eta_{n+1}$ in \eqref{eta-behavior}, and of
$\mathfrak{L}_{n+1}$ in Lemma~\ref{lemma:L} to obtain,
\[
   \tilde{\mathfrak{C}}_{n+1} \eta_{n+1}
      \leq \mathfrak{L}_0 \hat{\mathfrak{L}}_0
           \mathfrak{C}_{n+1}^2 \omega^6  \varphi(N_{n+1}) .
\]
Substituting the estimate of the stability constant in
Proposition~\ref{prop:Stab} into the above inequality, we get
\begin{equation}\label{est-h1}
   \tilde{\mathfrak{C}}_{n+1}\eta_{n+1}
      \leq \mathfrak{L}_0 \hat{\mathfrak{L}}_0 \omega^2
    e^{2K (1 + \omega^2 B_2) N_{n+1}^{\frac{4}{7}}}  \varphi(N_{n+1}) .
\end{equation}
We also find that
\begin{equation}\label{est-h2}
   \hat{\mathfrak{L}}_{n+1}
         \mathfrak{C}_{n+1} \tilde{\mathfrak{C}}_{n+1}
               ((3 + \varepsilon) \eta_n +\eta_{n+1} )
   \leq \hat{\mathfrak{L}}_0^2 \mathfrak{L}_0 \omega^2
      e^{3K (1 + \omega^2 B_2) N_{n+1}^{\frac{4}{7}}}
   \left[(3 + \varepsilon)  \varphi(N_n)
       +  \varphi(N_{n+1}) \right] .
\end{equation}
Using \eqref{est-h1}, we conclude that (\ref{eq:om1}) implies that the
first inequality in \eqref{level-cont} is satisfied; using
\eqref{est-h2}, we find that (\ref{eq:om2}) implies that the second
inequality in \eqref{level-cont} is satisfied.

\medskip\medskip

In our inverse problem, we fix the frequency $\omega$ and specify the
level by the domain partitioning. Our multi-level scheme starts at a
coarse domain partitioning. We let $\mathcal{D}_{N_0}$ be an initial
domain partitioning of $\Omega$ with the number of subdomains being
equal to $N_0$. This initial domain partitioning needs to allow a
large-scale approximation of the unique solution in
$\Omega$. Algorithm~\ref{algo:prog} is designed such that the starting
point $c^{-2}_{n+1,0}$ at level $n+1$, which equals $c^{-2}_{n,K_n}$,
is within the $(n+1)-$level radius of convergence, $\rho_{n+1}$
\cite{Hoop2012a}. Therefore, the iterations can continue until the the
accuracy limitation is reached. In the above, (\ref{level-cont})
yields a sufficient multi-level condition balancing the competition
between the approximation errors and the convergence radii of
neighboring levels \cite{Hoop2012a}.

\subsection*{Lowering the frequency}

In the next theorem, we prove that by letting the frequency go to
zero, the overall convergence radius tends to infinity. We note that
the statement in the theorem does not give a lower bound of the
frequency that is uniform for all unique solutions and starting
points. The frequency needs to satisfy the conditions in Propositions
\ref{uniformconstants} or \ref{continuity} in the absence of prior
knowledge of $\Sigma_{c_{\dagger}^{-2}}$.

\medskip\medskip

\begin{theorem}\label{thm:large-conv}
Let $\mathcal{D}_{N_0}$ be an arbitrary initial domain partitioning
and $y$ be the given data. Then, for any positive number $M_0$, there
exists a frequency $\omega_0$ low enough such that the convergence
radius $\rho \ge M_0$.
\end{theorem}

\begin{proof}
The proof follows by examining the behavior of the convergence radius
with respect to the frequency. All the constant are associated with
the initial level. For simplicity of notation, we omit the subscripts.
We note that $\rho$ in \eqref{abbreviation_2a} can be written as
\begin{equation}\label{eq:rholf}
   (2\rho)^{1/2} = \frac{1+\sqrt{1-8\tilde{\mathfrak{C}}\eta}
       - 4\tilde{\mathfrak{C}} \eta}{
             2\mathfrak{L}\mathfrak{C}^2\hat{\mathfrak{L}}} .
\end{equation}
The constant $\tilde{\mathfrak{C}}$ does not depend on
frequency. First, we provide a uniform lower bound for the numerator
\[
   1+\sqrt{1-8\tilde{\mathfrak{C}}\eta} - 4\tilde{\mathfrak{C}} \eta .
\]
From \eqref{cns} and \eqref{eta-behavior}, we conclude that
\[
   \tilde{\mathfrak{C}} \eta = \mathfrak{L}\mathfrak{C}^2 \eta
\leq \mathfrak{L}_0 \, e^{2K (1 + \omega^2 B_2) (|\Omega| / r_0^3)^{\frac{4}{7}}} \hat{\mathfrak{L}}_0 \omega^2 \varphi(N) .
\]
We choose $\omega$ sufficiently small such that $8\tilde{\mathfrak{C}}
\eta < 1$. Then,
\[
1+\sqrt{1-8\tilde{\mathfrak{C}}\eta} - 4\tilde{\mathfrak{C}} \eta = \frac{1}{2} (\sqrt{1-8\tilde{\mathfrak{C}}\eta} + 1)^2 \geq \frac{1}{2}.
\]
Futhermore, the denominator,
$\mathfrak{L}\mathfrak{C}^2\hat{\mathfrak{L}}$, in (\ref{eq:rholf})
satisfies the estimate,
\[
\mathfrak{L}\mathfrak{C}^2\hat{\mathfrak{L}} \leq
\mathfrak{L}_0  \hat{\mathfrak{L}}_0 \,  e^{2K (1 + \omega^2 B_2) (|\Omega| / r_0^3)^{\frac{4}{7}}} \omega^{2} ,
\]
following from \eqref{cns}. Hence, it tends to zero as the frequency
goes to zero. This completes the proof.
\end{proof}

\section{Discussion}

As an example of forming a multi-scale hierarchy, $W_{N_n}$, we
mention the use of Haar wavelets and `local' scale refinement. We
refer to convergence of our scheme from the following point of
view. Within one level, $n$, the iterates converge in the sense that
the distance between the iterate and the ball centered at the
$Z_n$-best approximation with a certain tolerance radius tends to
zero. This tolerance radius is determined by the approximation error.
We use a ball, instead of a point, to describe the convergence. This
is because if the iterates get sufficiently close to the $Z_n$-best
approximation, the descent direction will mainly drive the iterate out
of the `stable set' while, still, the projection will pull it back.

In the multi-level scheme, convergence means that the result
$c_{n,K_n}$ converges to the true solution, $c^{\dagger}$, as the
approximation error $\eta_n$ tends to $0$. Unless $\varphi$ decays
sufficiently fast, such a convergence is not guaranteed. We mention
that the high-frequency regime yields different
estimates~\footnote{Recently, increasing stability-type estimates for
  the inverse boundary value problem for the Helmholtz equation with
  increasing frequency have been obtained
  \cite{Nagayasu2013,Isakov2013}; however, the stability constants in
  these estimates do not enable the application of our scheme.} and we
expect that the multi-level condition will change accordingly.

\section*{Acknowledgements}

The authors thank the members, BGP, ExxonMobil, PGS, Statoil and
Total, of the Geo-Mathematical Imaging for partial support. The work
of E. Beretta was partially supported by MIUR grant PRIN
20089PWTPS003. The research of M.V. de Hoop and L. Qiu was supported
in part by National Science Foundation grant CMG DMS-102531. The work
of O. Scherzer has been supported by the Austrian Science Fund (FWF)
within the national research networks Photo\-acoustic Imaging in
Biology and Medicine, project S10505 and Geometry and Simulation
S11704. The paper has been completed during a stay at the Erwin
Schr\"odinger Institute (ESI) of M.V. de Hoop. The support of ESI is
appreciated.

\bibliographystyle{amsplain}

\bibliography{conv}

\end{document}